\RequirePackage{fix-cm}
\documentclass[smallcondensed]{svjour3}     
\smartqed  
\usepackage{graphicx}
\usepackage{amssymb,latexsym}

\begin{document}

\title{CLOSE-TO-REGULARITY AND COMPLETELY REGULARITY OF BOUNDED TRI-LINEAR MAPS
}
\subtitle{Do you have a subtitle?\\ If so, write it here}


\author{Abotaleb Sheikhali \and Ali Ebadian\and Kazem Haghnejad Azar 
}


\institute{A. Sheikhali \at
Department of Mathematics, Payame Noor University (PNU), Tehran, Iran.\\
\email{Abotaleb.sheikhali.20@gmail.com}
                                   \and
           A. Ebadian \at
              \email{Ebadian.ali@gmail.com}
               \and
           K. Haghnejad Azar \\
           Department of Mathematics, University of Mohaghegh Ardabili\\
\email{Haghnejad@uma.ac.ir}
}

\date{Received: date / Accepted: date}

\maketitle

\begin{abstract}
Let $f:X\times Y\times Z\longrightarrow W $ be a bounded tri-linear map on normed spaces. We say that $f$ is close-to-regular when $f^{t****s}=f^{s****t}$ and we say that $f$ is completely regular when all natural extensions are equal. In this manuscript, we have some results on the close-to-regular maps and investigate the close-to-regularity of tri-linear maps. We investigate the relation between Arens regularity of bounded bilinear maps and close-to-regularity bounded tri-linear maps. We give a simple criterion for the completely regularity of tri-linear maps.  We provide a necessary and sufficient condition such that the fourth adjoint $D^{****}$  of a tri-derivation is again a tri-derivation.

\end{abstract}

\section{Introduction}
\label{intro}
Let $X, Y, Z$ and $W$ be normed spaces and $f:X\times Y\times Z\longrightarrow W $ be a bounded tri-linear mapping. The natural extensions  of $f$ are  as following:
\begin{enumerate}
\item $f^{*}:W^{*}\times X\times Y\longrightarrow Z^{*}$, given by $\langle f^{*}(w^{*},x,y),z\rangle=\langle w^{*},f(x,y,z)\rangle$ where $x\in X, y\in Y, z\in Z, w^{*}\in W^{*}$. 

The map $f^*$ is a bounded tri-linear mapping and is said  the adjoint of $f$.

\item $f^{**}=(f^*)^*:Z^{**}\times W^{*}\times X\longrightarrow Y^{*}$, given by $\langle f^{**}(z^{**},w^{*},x),y\rangle=\langle z^{**}, f^{*}(w^{*},\\x,y)$ where $x\in X, y\in Y, z^{**}\in Z^{**}, w^{*}\in W^{*}$.

\item $f^{***}=(f^{**})^*:Y^{**}\times Z^{**}\times W^{*}\longrightarrow X^{*}$, given by $\langle f^{***}(y^{**},z^{**},w^{*}),x\rangle=\langle y^{**},  f^{**}(z^{**},w^{*},x) \rangle$ where $x\in X, y^{**}\in Y^{**}, z^{**}\in Z^{**}, w^{*}\in W^{*}$.

\item $f^{****}=(f^{***})^*:X^{**}\times Y^{**}\times Z^{**}\longrightarrow W^{**}$, given by $\langle f^{****}(x^{**},
y^{**},z^{**})$, $w^{*}\rangle =\langle x^{**}, f^{***}(y^{**},z^{**},w^{*}) \rangle$ where $x^{**}\in X^{**}, y^{**}\in Y^{**}, z^{**}\in Z^{**}, w^{*}\in W^{*}$.
\end{enumerate}
The bounded tri-linear map $f^{****}$ is the unique extension of $f$ such that the maps 
\begin{eqnarray*}
&&x^{**}\longrightarrow f^{****}(x^{**},y^{**},z^{**}):X^{**}\longrightarrow W^{**},\\
&&y^{**}\longrightarrow f^{****}(x,y^{**},z^{**}):Y^{**}\longrightarrow W^{**},\\
&&z^{**}\longrightarrow f^{****}(x,y,z^{**}):Z^{**}\longrightarrow W^{**},
\end{eqnarray*}
are weak$^{*}-$weak$^{*}$ continuous for each $x\in X, y\in Y, x^{**}\in X^{**}, y^{**}\in Y^{**}$ and $z^{**}\in Z^{**}$.
Now let 
\begin{eqnarray*}
&&f^i:Y\times X\times Z\longrightarrow W : f^i(y,x,z)=f(x,y,z),\\
&&f^j:X\times Z\times Y\longrightarrow W : f^j(x,z,y)=f(x,y,z),\\
&&f^r:Z\times Y\times X\longrightarrow W : f^r(z,y,x)=f(x,y,z),\\
&&f^t:Z\times X\times Y\longrightarrow W : f^t(z,x,y)=f(x,y,z),\\
&&f^s:Y\times Z\times X\longrightarrow W : f^s(y,z,x)=f(x,y,z),
\end{eqnarray*}
be the flip maps of $f$, for every $x\in X ,y\in Y$ and $z\in Z$. The flip maps of $f$ are bounded tri-linear maps. Also $f^{st}, f^{ts}, f^{ttt}$ and $f^{sss}$ are bounded tri-linear maps from $X\times Y\times Z$ to $W$.
It is easily seen that $f^{i****i}, f^{j****j}, f^{r****r}, f^{t****s}$ and $f^{s****t}$ are natural extensions of $f$ such that bounded linear operators
\begin{eqnarray*}
&&x^{**}\longrightarrow f^{i****i}(x^{**},y,z^{**}):X^{**}\longrightarrow W^{**}, \\
&&y^{**}\longrightarrow f^{i****i}(x^{**},y^{**},z^{**}):Y^{**}\longrightarrow W^{**},\\
&&z^{**}\longrightarrow f^{i****i}(x,y,z^{**}):Z^{**}\longrightarrow W^{**},\\
&&x^{**}\longrightarrow f^{j****j}(x^{**},y,z^{**}):X^{**}\longrightarrow W^{**}, \\
&&y^{**}\longrightarrow f^{j****j}(x,y^{**},z):Y^{**}\longrightarrow W^{**},\\
&&z^{**}\longrightarrow f^{j****j}(x,y^{**},z^{**}):Z^{**}\longrightarrow W^{**},\\
&&x^{**}\longrightarrow f^{r****r}(x^{**},y,z):X^{**}\longrightarrow W^{**}, \\
&&y^{**}\longrightarrow f^{r****r}(x^{**},y^{**},z):Y^{**}\longrightarrow W^{**},\\
&&z^{**}\longrightarrow f^{r****r}(x^{**},y^{**},z^{**}):Z^{**}\longrightarrow W^{**},\\
&&x^{**}\longrightarrow f^{t****s}(x^{**},y^{**},z):X^{**}\longrightarrow W^{**}, \\
&&y^{**}\longrightarrow f^{t****s}(x,y^{**},z):Y^{**}\longrightarrow W^{**},\\
&&z^{**}\longrightarrow f^{t****s}(x^{**},y^{**},z^{**}):Z^{**}\longrightarrow W^{**},\\
&&x^{**}\longrightarrow f^{s****t}(x^{**},y,z):X^{**}\longrightarrow W^{**}, \\
&&y^{**}\longrightarrow f^{s****t}(x^{**},y^{**},z^{**}):Y^{**}\longrightarrow W^{**},\\
&&z^{**}\longrightarrow f^{s****t}(x^{**},y,z^{**}):Z^{**}\longrightarrow W^{**},
\end{eqnarray*}
are weak$^{*}-$weak$^{*}$ continuous for each $x\in X, y\in Y, z\in Z, x^{**}\in X^{**}, y^{**}\in Y^{**}$ and $z^{**}\in Z^{**}$.
For natural extensions of $f$ we have
\begin{enumerate}
\item $f^{i****i}(x^{**},y^{**},z^{**})=w^{*}-\lim\limits_\beta w^{*}-\lim\limits_\alpha w^{*}-\lim\limits_\gamma f(x_\alpha,y_\beta,z_{\gamma})$,

\item $f^{j****j}(x^{**},y^{**},z^{**})=w^{*}-\lim\limits_\alpha w^{*}-\lim\limits_\gamma w^{*}-\lim\limits_\beta f(x_\alpha,y_\beta,z_{\gamma})$,

\item $f^{r****r}(x^{**},y^{**},z^{**})=w^{*}-\lim\limits_\gamma w^{*}-\lim\limits_\beta w^{*}-\lim\limits_\alpha f(x_\alpha,y_\beta,z_{\gamma})$,

\item $f^{****}(x^{**},y^{**},z^{**})=w^{*}-\lim\limits_\alpha w^{*}-\lim\limits_\beta w^{*}-\lim\limits_\gamma f(x_\alpha,y_\beta,z_{\gamma})$,

\item $f^{t****s}(x^{**},y^{**},z^{**})=w^{*}-\lim\limits_\gamma w^{*}-\lim\limits_\alpha w^{*}-\lim\limits_\beta f(x_\alpha,y_\beta,z_{\gamma})$,

\item $f^{s****t}(x^{**},y^{**},z^{**})=w^{*}-\lim\limits_\beta w^{*}-\lim\limits_\gamma w^{*}-\lim\limits_\alpha f(x_\alpha,y_\beta,z_{\gamma}),$
\end{enumerate}
where $\{x_{\alpha} \}, \{y_{\beta} \}$ and $\{z_{\gamma} \}$ are nets in $X, Y$ and $Z$  which converge to $x^{**}\in X^{**},y^{**}\in Y^{**}$ and $z^{**}\in Z^{**}$  in the $w^{*}-$topologies, respectively. 
\begin{definition}\label{definition1}
A bounded tri-linear map $f$ is said to be completely regular when all natural extensions are equal, that is, $f^{i****i}= f^{j****j}=f^{r****r}=f^{****}=f^{t****s}=f^{s****t}$  holds. Also $f$ is said to be close-to-regular if $f^{t****s}=f^{s****t}$. If $f$ is completely regular, then trivially $f$ is close-to-regular. It is obvious that  $f$ is close-to-regular if and only if $f^{s*****s}=f^{t******j}.$
\end{definition}
Throughout the article, we usually identify a normed space with its canonical image in its second dual.

\section{Close-to-regular maps}
\label{se2}

We commence with the following theorem for  close-to-regular maps.
\begin{theorem}\label{theorem1}
For a bounded tri-linear map $f:X\times Y\times Z\longrightarrow W$  the following statements are equivalent:
\begin{enumerate}
\item $f$ is close-to-regular.
\item $f^{s***t*}(Y^{**},W^{*},Z)\subseteq X^{*}$ and $f^{s******}(X^{**},W^{*},Y^{**})\subseteq Z^{*}$.
\item $f^{t*****}(W^{*},Z^{**},X^{**})\subseteq Y^{*}$.
\end{enumerate}
\end{theorem}
\begin{proof}
Suppos $\{x_{\alpha} \}, \{y_{\beta} \}$ and $\{z_{\gamma} \}$ are nets in $X, Y$ and $Z$  which converge to $x^{**}\in X^{**},y^{**}\in Y^{**}$ and $z^{**}\in Z^{**}$  in the $w^{*}-$topologies, respectively.

(1) $\Rightarrow$ (2), if $f$ is close-to-regular, then $f^{t****s}=f^{s****t}$. For every $x^{**}\in X^{**},y^{**}\in Y^{**},z\in Z$ and $w^{*}\in W^{*}$ we have 
\begin{eqnarray*}
&\langle & f^{s***t*}(y^{**},w^{*},z),x^{**}\rangle=\langle y^{**},f^{s***}(z,x^{**},w^{*})\rangle\\
&=&\langle f^{s****t}(x^{**},y^{**},z),w^{*}\rangle=\langle f^{t****s}(x^{**},y^{**},z),w^{*}\rangle\\
&=&\langle f^{t****}(z,x^{**},y^{**}),w^{*}\rangle=\langle f^{t**}(y^{**},w^{*},z),x^{**}\rangle.
\end{eqnarray*}
Therefore $f^{s***t*}(y^{**},w^{*},z)=f^{t**}(y^{**},w^{*},z)\in X^{*}$, follows that $f^{s***t*}(Y^{**},W^{*},Z)\subseteq X^{*}$. In the other hand,
 \begin{eqnarray*}
 &\langle& f^{s******}(x^{**},w^{*},y^{**}),z^{**}\rangle=\langle w^{*},f^{s****}(y^{**},z^{**},x^{**})\rangle\\
&=&\langle w^{*},f^{s****t}(x^{**},y^{**},z^{**})\rangle=\langle w^{*},f^{t****s}(x^{**},y^{**},z^{**})\rangle\\
&=&\langle w^{*},f^{t****}(z^{**},x^{**},y^{**})\rangle=\langle z^{**},f^{t***}(x^{**},y^{**},w^{*})\rangle.
\end{eqnarray*}
Since the $f^{t***}(x^{**},y^{**},w^{*})\in z^{*}$, thus $f^{s******}(X^{**},W^{*},Y^{**})\subseteq Z^{*}$, as claimed.

(2) $\Rightarrow$ (3), if (2) holds then
 \begin{eqnarray*}
&\langle &f^{t*****}(w^{*},z^{**},x^{**}),y^{**}\rangle=\lim\limits_\gamma\lim\limits_\alpha\lim\limits_\beta \langle f(x_\alpha,y_\beta,z_{\gamma}),w^{*} \rangle\\
&=&\lim\limits_\gamma\lim\limits_\alpha\lim\limits_\beta \langle w^{*},f^{s}(y_{\beta},z_{\gamma},x_{\alpha})\rangle=\lim\limits_\gamma\lim\limits_\alpha\lim\limits_\beta \langle f^{s***}(z_{\gamma},x_{\alpha},w^{*}),y_{\beta}\rangle\\
&=&\lim\limits_\gamma\lim\limits_\alpha\langle y^{**},f^{s***t}(w^{*},z_{\gamma},x_{\alpha})\rangle=\lim\limits_\gamma\lim\limits_\alpha\langle f^{s***t*}(y^{**},w^{*},z_{\gamma}),x_{\alpha}\rangle\\
&=&\lim\limits_\gamma\langle f^{s***t*}(y^{**},w^{*},z_{\gamma}),x^{**}\rangle=\lim\limits_\gamma\langle y^{**},f^{s***t}(w^{*},z_{\gamma},x^{**})\rangle\\
&=&\lim\limits_\gamma\langle y^{**},f^{s***}(z_{\gamma},x^{**},w^{*})\rangle=\lim\limits_\gamma\langle f^{s******}(x^{**},w^{*},y^{**}),z_{\gamma}\rangle\\
&=&\langle f^{s******}(x^{**},w^{*},y^{**}),z^{**}\rangle=\langle f^{s***}(z^{**},x^{**},w^{*}),y^{**} \rangle\\
&=&\langle f^{s***t}(w^{*},z^{**},x^{**}),y^{**} \rangle.
\end{eqnarray*}
Since  $f^{s***t}(w^{*},z^{**},x^{**})\in Y^{*}$, thus (3) holds.

(3) $\Rightarrow$ (1), let $f^{t*****}(W^{*},Z^{**},X^{**})\subseteq Y^{*}$. Then for every $w^{*}\in W^{*}$ we have, 
 \begin{eqnarray*}
&\langle & f^{s****t}(x^{**},y^{**},z^{**}),w^{*}\rangle=\lim\limits_\beta\lim\limits_\gamma\lim\limits_\alpha \langle f(x_\alpha,y_\beta,z_{\gamma}),w^{*}\rangle\\
&=&\lim\limits_\beta\lim\limits_\gamma\lim\limits_\alpha \langle w^{*},f^{t}(z_{\gamma},x_{\alpha},y_{\beta}\rangle=\lim\limits_\beta\lim\limits_\gamma\lim\limits_\alpha \langle f^{t*}(w^{*},z_{\gamma},x_{\alpha}),y_{\beta}\rangle\\
&=&\lim\limits_\beta\lim\limits_\gamma\lim\limits_\alpha \langle f^{t**}(y_{\beta},w^{*},z_{\gamma}),x_{\alpha}\rangle=\lim\limits_\beta\lim\limits_\gamma \langle x^{**},f^{t**}(y_{\beta},w^{*},z_{\gamma})\rangle\\
&=&\lim\limits_\beta\lim\limits_\gamma \langle f^{t***}( x^{**},y_{\beta},w^{*}),z_{\gamma}\rangle=\lim\limits_\beta \langle z^{**},f^{t***}( x^{**},y_{\beta},w^{*})\rangle\\
&=&\lim\limits_\beta \langle f^{t****}(z^{**}, x^{**},y_{\beta}),w^{*}\rangle=\lim\limits_\beta \langle f^{t*****}(w^{*},z^{**}, x^{**}),y_{\beta}\rangle\\
&=&\langle f^{t*****}(w^{*},z^{**}, x^{**}),y^{**}\rangle=\langle f^{t****s}(x^{**},y^{**},z^{**}),w^{*}\rangle.
\end{eqnarray*}
It follows that $f$ is close-to-regular and this completes the proof.
\end{proof}
As an immediate consequence of Theorem \ref{theorem1}, we deduce the next result.
\begin{corollary}\label{corollary1}
Let $f:X\times Y\times Z\longrightarrow W $ be a bounded tri-linear mapping.
\begin{enumerate}
\item If $Y$ is reflexive, then $f$ is close-to-regular.
\item If $X$ and $Z$ are reflexive, then $f$ is close-to-regular.
\end{enumerate}
\end{corollary}
\begin{example}\label{example1}
Let $G$ be a finite locally compact Hausdorff group. We know from \cite{young} that $L^{1}(G)$ is regular if and only if it is reflexive or  $G$ is finite. So the bounded tri-linear mapping $f:L^{1}(G)\times L^{1}(G)\times L^{1}(G)\longrightarrow L^{1}(G)$ defined by $f(k,g,h)=k* g* h$ is close-to-regular, where $(k*g)(x)=\int_{G}k(y)g(y^{-1}x)dy$ for every $k, g$ and $h\in L^{1}(G)$.  
\end{example}
\begin{theorem}\label{theorem2}
Let $f:X\times Y\times Z\longrightarrow W$ be a bounded tri-linear map. Then,
\begin{enumerate}
\item $f^{r}$ is close-to-regular  if and only if $f^{i****i}=f^{j****j}$.
\item $f^{i}$ is close-to-regular  if and only if $f^{j****j}=f^{r****r}$.
\item $f^{j}$ is close-to-regular  if and only if $f^{i****i}=f^{r****r}$.
\item $f^{t}$ is close-to-regular  if and only if $f^{s****t}=f^{****}$.
\item $f^{s}$ is close-to-regular  if and only if $f^{t****s}=f^{****}$.
\end{enumerate}
\end{theorem}
\begin{proof}
We prove only (1), the other parts have the same argument. Let $x^{**}\in X^{**},y^{**}\in Y^{**}, z^{**}\in Z^{**}$ and $w^{*}\in W^{*}$ and let $\{x_{\alpha} \}, \{y_{\beta} \}$ and $\{z_{\gamma} \}$ be nets in $X, Y$ and $Z$  which converge to $x^{**}, y^{**}$ and $z^{**}$  in the $w^{*}-$topologies, respectively. Then we have
\begin{eqnarray*}
\langle f^{i****i}(x^{**},y^{**},z^{**}),w^{*}\rangle &=&\lim\limits_{\beta}\lim\limits_{\alpha}\lim\limits_{\gamma}\langle f(x_{\alpha},y_{\beta},z_{\gamma}),w^{*}\rangle\\
&=&\lim\limits_{\beta}\lim\limits_{\alpha}\lim\limits_{\gamma}\langle f^{r}(z_{\gamma},y_{\beta},x_{\alpha}),w^{*}\rangle\\
&=&\langle f^{rs****t}(z^{**},y^{**},x^{**}),w^{*}\rangle.
\end{eqnarray*}
Therefore $f^{i****i}=f^{rs****t}$. In the other hand 
\begin{eqnarray*}
\langle f^{j****j}(x^{**},y^{**},z^{**}),w^{*}\rangle &=&\lim\limits_{\alpha}\lim\limits_{\gamma}\lim\limits_{\beta}\langle f(x_{\alpha},y_{\beta},z_{\gamma}),w^{*}\rangle\\
&=&\lim\limits_{\alpha}\lim\limits_{\gamma}\lim\limits_{\beta}\langle f^{r}(z_{\gamma},y_{\beta},x_{\alpha}),w^{*}\rangle\\
&=&\langle f^{rt****s}(z^{**},y^{**},x^{**}),w^{*}\rangle.
\end{eqnarray*}
Thus $f^{j****j}=f^{rt****s}$ and this completes the proof.
\end{proof}
As immediate consequences of the Theorem \ref{theorem2} we have the next corollaries.
\begin{corollary}\label{corollary2}
If $f$ is completely regular, then $f^{i},f^{j},f^{r}, f^{t}$ and $f^{s}$ are close-to-regular.
\end{corollary}
\begin{corollary}\label{corollary3}
If $f^{s}$ and $f^{t}$ are close-to-regular, then $f$ is close-to-regular.
\end{corollary}
\begin{theorem}\label{theorem3}
Let $f:X\times Y\times Z\longrightarrow W$ and $g:X\times S\times Z\longrightarrow W$ be
bounded tri-linear mappings and let $h:Y\longrightarrow S$ be a bounded linear mapping such that
$f(x,y,z)=g(x,h(y),z)$, for every $x\in X, y\in Y$ and $z\in Z$. If $h$ is weakly compact, then $f$ is close-to-regular.
\end{theorem}
\begin{proof}
Suppos $\{x_{\alpha} \}, \{y_{\beta} \}$ and $\{z_{\gamma} \}$ are nets in $X, Y$ and $Z$  which converge to $x^{**}\in X^{**},y^{**}\in Y^{**}$ and $z^{**}\in Z^{**}$  in the $w^{*}-$topologies, respectively. Then a direct verification reveals that
$f^{t****s}(x^{**},y^{**},z^{**})=g^{t****s}(x^{**},h^{**}(y^{**}),z^{**}).$
Then for each $y^{**}\in Y^{**}$ we have
\begin{eqnarray*}
\langle f^{t*****}(w^{*},z^{**},x^{**}),y^{**}\rangle &=&\langle w^{*},f^{t****}(z^{**},x^{**},y^{**})\rangle\\
&=& \langle w^{*},f^{t****s}(x^{**},y^{**},z^{**})\rangle\\
&=& \langle w^{*},g^{t****s}(x^{**},h^{**}(y^{**}),z^{**})\rangle\\
&=& \langle w^{*},g^{t****}(z^{**},x^{**},h^{**}(y^{**}))\rangle\\
&=& \langle g^{t*****}(w^{*},z^{**},x^{**}),h^{**}(y^{**})\rangle\\
&=& \langle h^{***}(g^{t*****}(w^{*},z^{**},x^{**})),y^{**}\rangle.
\end{eqnarray*}
Therefore $f^{t*****}(w^{*},z^{**},x^{**})=h^{***}(g^{t*****}(w^{*},z^{**},x^{**}))$. The weak compactness of $h$ implies that of $h^{*}$, from which we have $h^{***}(S^{***})\subseteq Y^{*}$. In particular,
$h^{***}(g^{t*****}(W^{*},Z^{**},X^{**}))\subseteq Y^{*}$, thus we deduce $f^{t*****}(W^{*},Z^{**},X^{**})\subseteq Y^{*}$. It follows that $f$ is close-to-regular and this completes the proof.
\end{proof}
If $Y$ or $S$ is reflexive, then every bounded linear mapping $h:Y\longrightarrow S$ is weakly compact. Thus we give the next result.
\begin{corollary}\label{corollary5}
Let $f:X\times Y\times Z\longrightarrow W$ and $g:X\times S\times Z\longrightarrow W$ be
bounded tri-linear mappings and let $h:Y\longrightarrow S$ be a bounded linear mapping such that
$f(x,y,z)=g(x,h(y),z)$, for every $x\in X, y\in Y$ and $z\in Z$. If $S$ is reflexive, then $f$ is close-to-regular.
\end{corollary}
\begin{theorem}\label{theorem4}
Let $f:X\times Y\times Z\longrightarrow W$ be bounded tri-linear mapping. If  $f^{****t**s}=f^{t**s****}$ and $f^{****s**t}=f^{s**t****}$. Then $f$ is close-to-regular
\end{theorem}
\begin{proof}
Let  $\{y_{\beta} \}\subseteq Y$ be net which converge to $y^{**}\in Y^{**}$ in the $w^{*}-$topologies. Using the equality $f^{****s**t}=f^{s**t****}$, a standard argument applies to show that $f^{****}=f^{s****t}$. In the other hand, let $\{z_{\gamma} \}\subseteq Z$ be net which converge to $z^{**}\in Z^{**}$  in the $w^{*}-$topologies. Then equality $f^{****t**s}=f^{t**s****}$ impleas that $f^{****}=f^{t****s}$. Therefore $f^{s****t}=f^{t****s}$, as claimed.
\end{proof}

\section{Relationship between close-to-regularity bounded tri-linear maps and Arens regularity bounded bilinear maps}
\label{se3}
Let $m:X\times Y\longrightarrow Z$ be a continuous bilinear map. The extension of $m$ on normed spaces and the concept of regularity of $m$ were introduced by Richard Arens in \cite{Arens1} and \cite{Arens2}. Some characterizations for the Arens regularity of bounded bilinear map $m$ and Banach algebra $A$ are proved in \cite{Arikan1}, \cite{Arikan2}, \cite{Civin and Yood}, \cite{Haghnejad Azar},\cite{ulger}, \cite{sheikh} and \cite{young}.

S. Mohammadzade and H. R. E. Vishki in \cite[Theorem 2.1]{Mohamadzadeh} (see also, \cite[Proposition 4.1]{Dales} ) have shown that, $m^{****}(Z^{*},X^{**})\subseteq Y^{*}$ if and only if  $m$ is Arens regular. Now applying Theorem \ref{theorem1} and we give next Theorem. 
\begin{theorem}\label{theorem5}
Let $f:X\times Y\times Z\longrightarrow W$ be a bounded tri-linear mapping and let $m:Y\times Z\longrightarrow X^{*}$ be a bounded bilinear mapping such that $m(y,z)=f^{s*}(w^{*},y,z)$, for every $y\in Y, z\in Z$ and $w^{*}\in W^{*}$ . If $f$ is close-to-regular, then $m$ is Arens regular.
\end{theorem}
\begin{proof}
Let $\{y_{\beta} \}$ and $\{z_{\gamma} \}$ be two net in $Y$ and $Z$  which converge to $y^{**}$ and $z^{**}$  in the $w^{*}-$topologies, respectively. Using the equality $m(y,z)=f^{s*}(w^{*},y,z)$, we have
\begin{eqnarray*}
\langle f^{s******}(x^{**},w^{*},y^{**}),z^{**}\rangle &=&\lim\limits_{\beta}\lim\limits_{\gamma}\langle x^{**},f^{s*}(w^{*},y_{\beta},z_{\gamma})\rangle\\
&=& \lim\limits_{\beta}\lim\limits_{\gamma}\langle x^{**},m(y_{\beta},z_{\gamma})\rangle\\
&=&\langle x^{**},m^{***}(y^{**},z^{**})\rangle\\
&=&\langle m^{****}(x^{**},y^{**}),z^{**}\rangle.
\end{eqnarray*}
It show that $f^{s******}(X^{**},W^{*},Y^{**})=m^{****}(X^{**},Y^{**})$. The close-to-regular of $f$ implies that $f^{s******}(X^{**},W^{*},Y^{**})\subseteq Z^{*}$. Therefore $m^{****}(X^{**},Y^{**})\subseteq Z^{*}$, as claimed.
\end{proof}
In the Theorem \ref{theorem5}, if $m$ is Arens regular, then a direct verification reveals that $f^{s****t}=f^{r****r}$.
\begin{theorem}\label{theorem6}
Let $f:X\times Y\times Z\longrightarrow W$ be a bounded tri-linear mapping and let $m:X\times Y\longrightarrow Z$ and $\theta:X\times Y\longrightarrow W$ be  bounded bilinear mappings such that $\theta(x,y)=f(x,y,m(x,y))$, for every $x\in X$ and $y\in Y$ . If $f^{t***r}=f^{r***t}$ and $f^{****t**s}=f^{t**s****}$, then $\theta$ is Arens regular.
\end{theorem}
\begin{proof}
Using the equality $f^{t***r}=f^{r***t}$, we have
\begin{eqnarray*}
\langle f^{t****s}(x^{**},y^{**},z^{**}),w^{*}\rangle &=&\langle z^{**},f^{t***}(x^{**},y^{**},w^{*})\rangle\\
&=&\langle z^{**},f^{t***r}(w^{*},y^{**},x^{**})\rangle\\
&=&\langle z^{**},f^{r***t}(w^{*},y^{**},x^{**})\rangle\\
&=&\langle z^{**},f^{r***}(y^{**},x^{**},w^{*},)\rangle\\
&=&\langle f^{r****r}(x^{**},y^{**},z^{**}),w^{*}\rangle.
\end{eqnarray*}
It follows that $f^{t****s}=f^{r****r}$. In the other hand, using the equality $f^{****t**s}=f^{t**s****}$ we have 
\begin{eqnarray*}
\langle f^{t****s}(x^{**},y^{**},z^{**}),w^{*}\rangle &=&\langle z^{**},f^{t***}(x^{**},y^{**},w^{*})\rangle\\
&=&\lim\limits_{\gamma}\langle f^{t***}(x^{**},y^{**},w^{*}),z_{\gamma}\rangle\\
&=&\lim\limits_{\gamma}\langle x^{**},f^{t**s}(w^{*},z_{\gamma},y^{**})\rangle\\
&=&\lim\limits_{\gamma}\langle f^{t**s**}(y^{**},x^{**},w^{*}),z_{\gamma} \rangle\\
&=&\langle z^{**},f^{t**s**}(y^{**},x^{**},w^{*})\rangle\\
&=&\langle f^{t**s****}(w^{*},z^{**},y^{**}),x^{**})\rangle\\
&=&\langle f^{****t**s}(w^{*},z^{**},y^{**}),x^{**})\rangle\\
&=&\langle f^{****t**}(y^{**},w^{*},z^{**}),x^{**})\rangle\\
&=&\langle f^{****}(x^{**},y^{**},z^{**}),w^{*})\rangle.
\end{eqnarray*}
Thus $f^{t****s}=f^{****}$. Therefore $f^{r****r}=f^{****}$. Now we have
\begin{eqnarray*}
\langle \theta^{***}(x^{**},y^{**}),w^{*}\rangle &=&\lim\limits_{\alpha}\lim\limits_{\beta}\langle w^{*},\theta(x_{\alpha},y_{\beta})\rangle\\
&=&\lim\limits_{\alpha}\lim\limits_{\beta}\langle w^{*},f(x_{\alpha},y_{\beta},m(x_{\alpha},y_{\beta}))\rangle\\
&=&\langle f^{****}(x^{**},y^{**},m^{***}(x_{\alpha},y_{\beta})),w^{*}\rangle\\
&=&\langle f^{r****r}(x^{**},y^{**},m^{***}(x_{\alpha},y_{\beta})),w^{*}\rangle\\
&=&\lim\limits_{\beta}\lim\limits_{\alpha}\langle w^{*},f(x_{\alpha},y_{\beta},m(x_{\alpha},y_{\beta}))\rangle\\
&=&\lim\limits_{\beta}\lim\limits_{\alpha}\langle w^{*},\theta(x_{\alpha},y_{\beta})\rangle=\langle \theta^{r***r}(x^{**},y^{**}),w^{*}\rangle.
\end{eqnarray*}
It follows that $\theta$ is Arens regular.
\end{proof}

\section{Completely regular maps}
\label{se4}

In this section, we provide a criterion for the completely regularity of a bounded tri-linear map.
\begin{theorem}\label{theorem7}
For a bounded tri-linear map $f:X\times Y\times Z\longrightarrow W$  the following statements are equivalent:
\begin{enumerate}
\item $f$ is completely regular.
\item $f^{*****}(W^{*},X^{**},Y^{**})\subseteq Z^{*}$ and $f^{******}(Z^{**},W^{*},X^{**})\subseteq Y^{*}$.
\item $f^{i*****}(W^{*},Y^{**},X^{**})\subseteq Z^{*}$ and $f^{i******}(Z^{**},W^{*},Y^{**})\subseteq X^{*}$.
\item $f^{j*****}(W^{*},X^{**},Z^{**})\subseteq Y^{*}$ and $f^{j******}(Y^{**},W^{*},X^{**})\subseteq Z^{*}$.
\item $f^{r*****}(W^{*},Z^{**},Y^{**})\subseteq X^{*}$ and $f^{r******}(X^{**},W^{*},Z^{**})\subseteq Y^{*}$.
\item $f^{t*****}(W^{*},Z^{**},X^{**})\subseteq Y^{*}$ and $f^{t******}(Y^{**},W^{*},Z^{**})\subseteq X^{*}$.
\item $f^{s*****}(W^{*},Y^{**},Z^{**})\subseteq X^{*}$ and $f^{s******}(X^{**},W^{*},Y^{**})\subseteq Z^{*}$.
\end{enumerate}
\end{theorem}
\begin{proof}
We prove only $(1)\Rightarrow (2)$ and $(2)\Rightarrow (1)$ ,the other parts $(1)\Rightarrow (3)$ and $(3)\Rightarrow (1)$ , ..., $(1)\Rightarrow (7)$ and $(7)\Rightarrow (1)$ have the same argument.

(1) $\Rightarrow$ (2) Since $f$ is competetly regular, thus all natural extensions are equal. In particular, $f^{****}=f^{i****i}$. So for every $x^{**}\in X^{**}, y^{**}\in Y^{**}, z^{**}\in Z^{**}$ and $w^{*}\in W^{*}$ we have
\begin{eqnarray*}
&\langle & f^{******}(z^{**},w^{*},x^{**}),y^{**}\rangle=\langle z^{**},f^{*****}(w^{*},x^{**},y^{**})\rangle\\
&=&\langle w^{*},f^{****}(x^{**},y^{**},z^{**})\rangle=\langle w^{*},f^{i****i}(x^{**},y^{**},z^{**})\rangle\\
&=&\langle f^{i****}(y^{**},x^{**},z^{**}),w^{*}\rangle=\langle f^{i***}(x^{**},z^{**},w^{*}),y^{**}\rangle\\
&=&\langle f^{i***s}(z^{**},w^{*},x^{**}),y^{**}\rangle.
\end{eqnarray*}
Therefore $f^{******}=f^{i***s}$. Since $f^{i***s}(Z^{**},W^{*},X^{**})\subseteq Y^{*}$, thus $f^{******}(Z^{**},\\W^{*},X^{**})\subseteq Y^{*}$. Similarly, if $f^{****}=f^{r****r}$ then $f^{*****}=f^{r***r}$ and since $f^{r***r}(W^{*},X^{**},Y^{**})\subseteq Z^{*}$ thus $f^{*****}(W^{*},X^{**},Y^{**})\subseteq Z^{*}.$\\
(2) $\Rightarrow$ (1), let (2) holds and let $\{x_{\alpha} \}, \{y_{\beta} \}$ and $\{z_{\gamma} \}$ be nets in $X, Y$ and $Z$  which converge to $x^{**}\in X^{**},y^{**}\in Y^{**}$ and $z^{**}\in Z^{**}$  in the $w^{*}-$topologies, respectively. Then we have
 \begin{eqnarray*}
&\langle &f^{s****t}(x^{**},y^{**},z^{**}),w^{*}\rangle=\lim\limits_\beta\lim\limits_\gamma\lim\limits_\alpha \langle f(x_\alpha,y_\beta,z_{\gamma}),w^{*} \rangle\\
&=&\lim\limits_\beta\lim\limits_\gamma\lim\limits_\alpha \langle f^{*}(w^{*},x_{\alpha},y_{\beta}),z_{\gamma})\rangle=\lim\limits_\beta\lim\limits_\gamma\lim\limits_\alpha \langle f^{**}(z_{\gamma},w^{*},x_{\alpha}),y_{\beta}\rangle\\
&=&\lim\limits_\beta\lim\limits_\gamma\lim\limits_\alpha\langle f^{***}(y_{\beta},z_{\gamma},w^{*}),x_{\alpha}\rangle=\lim\limits_\beta\lim\limits_\gamma\langle x^{**},f^{***}(y_{\beta},z_{\gamma},w^{*}) \rangle\\
&=&\lim\limits_\beta\lim\limits_\gamma\langle f^{****}(x^{**},y_{\beta},z_{\gamma}),w^{*}\rangle=\lim\limits_\beta\lim\limits_\gamma\langle f^{*****}(w^{*},x^{**},y_{\beta}),z_{\gamma}\rangle\\
&=&\lim\limits_\beta\langle z^{**},f^{*****}(w^{*},x^{**},y_{\beta})\rangle=\lim\limits_\beta\langle f^{******}(z^{**},w^{*},x^{**}),y_{\beta}\rangle\\
&=&\langle f^{******}(z^{**},w^{*},x^{**}),y^{**}\rangle=\langle z^{**},f^{*****}(w^{*},x^{**},y^{**})\\
&=&\langle f^{****}(x^{**},y^{**},z^{**}),w^{*}\rangle.
\end{eqnarray*}
It follows that $f^{s****t}=f^{****}$. A similar argument applies equality $f^{t****s}=f^{****}, f^{i****i}=f^{****}, f^{j****j}=f^{****}$ and $f^{r****r}=f^{****}$. Therefore $f$ is completely regular.
\end{proof}
As immediate consequences of the Theorem \ref{theorem7} we have the next corollaries.
\begin{corollary}\label{corollary6}
A bounded tri-linear map $f$ is completely regular if and only if $f^{i}$ (or $f^{j}$ or $f^{r}$ or $f^{t}$ or $f^{s}$) is completely regular.
\end{corollary}
\begin{corollary}\label{corollary7}
If from $X, Y$ or $Z$ at least two reflexive then $f$ is completely regular.
\end{corollary}
\begin{corollary}\label{corollary8}
If $X$ is reflexive then the bounded tri-linear map $g:X\times X\times X\longrightarrow X$ is completely regular.
\end{corollary}
\begin{corollary}\label{corollary9}
Let $X, Y, Z, W$ and $S$ be Banach spaces and let $f:X\times Y\times Z\longrightarrow W$, $g:X\times S\times Z\longrightarrow W$ and $K:X\times Y\times S\longrightarrow W$ be bounded tri-linear mappings and $h_{1}:Y\longrightarrow S, h_{2}:Z\longrightarrow S$ be bounded linear mappings shuch that $$f(x,y,z)=K(x,y,h_{2}(z))=g(x,h_{1}(y),z).$$
If $h_{1}$ and $h_{2}$ are weakly compact then $f$ is completely regular.
\end{corollary}
\begin{proof}
Let $\{x_{\alpha} \}, \{y_{\beta} \}$ and $\{z_{\gamma} \}$ are nets in $X, Y$ and $Z$  which converge to $x^{**}\in X^{**},y^{**}\in Y^{**}$ and $z^{**}\in Z^{**}$  in the $w^{*}-$topologies, respectively. Then we have 
 \begin{eqnarray*}
\langle f^{*****}(w^{*},x^{**},y^{**}),z^{**}\rangle &=&\langle f^{****}(x^{**},y^{**},z^{**}),w^{*} \rangle\\
&=&\lim\limits_\alpha\lim\limits_\beta\lim\limits_\gamma \langle f(x_\alpha,y_\beta,z_{\gamma}),w^{*} \rangle\\
&=&\lim\limits_\alpha\lim\limits_\beta\lim\limits_\gamma \langle K(x_\alpha,y_\beta,h_{2}(z_{\gamma})),w^{*} \rangle\\
&=&\lim\limits_\alpha\lim\limits_\beta\lim\limits_\gamma \langle K^{*}(w^{*},x_{\alpha},y_{\beta}),h_{2}(z_{\gamma}) \rangle\\
&=&\lim\limits_\alpha\lim\limits_\beta\lim\limits_\gamma \langle h_{2}^{*}(K^{*}(w^{*},x_{\alpha},y_{\beta})),z_{\gamma}\rangle\\
&=&\lim\limits_\alpha\lim\limits_\beta \langle z^{**},h_{2}^{*}(K^{*}(w^{*},x_{\alpha},y_{\beta}))\rangle\\
&=&\lim\limits_\alpha\lim\limits_\beta \langle h_{2}^{**}(z^{**}),K^{*}(w^{*},x_{\alpha},y_{\beta})\rangle\\
&=&\lim\limits_\alpha\lim\limits_\beta \langle K^{**}(h_{2}^{**}(z^{**}),w^{*},x_{\alpha}),y_{\beta}\rangle\\
&=&\lim\limits_\alpha \langle y^{**},K^{**}(h_{2}^{**}(z^{**}),w^{*},x_{\alpha})\rangle\\
&=&\lim\limits_\alpha \langle K^{***}(y^{**},h_{2}^{**}(z^{**}),w^{*}),x_{\alpha}\rangle\\
&=&\langle x^{**},K^{***}(y^{**},h_{2}^{**}(z^{**}),w^{*})\rangle\\
&=&\langle K^{****}(x^{**},y^{**},h_{2}^{**}(z^{**})),w^{*}\rangle\\
&=&\langle K^{*****}(w^{*},x^{**},y^{**}),h_{2}^{**}(z^{**})\rangle\\
&=&\langle h_{2}^{***}(K^{*****}(w^{*},x^{**},y^{**})),z^{**}\rangle.
\end{eqnarray*}
Thus $f^{*****}(W^{*},X^{**},Y^{**})=h_{2}^{***}(K^{*****}(W^{*},X^{**},Y^{**}))$. Since the weak compactness of $h_{2}$ implies that of $h_2^{*}$, from which we have $h_{2}^{***}(S^{***})\subseteq Z^{*}$. In particular,
$h_{2}^{***}(K^{*****}(W^{*},X^{**},Y^{**}))\subseteq Z^{*}$, thus we deduce $f^{*****}(W^{*},X^{**},Y^{**})\subseteq Z^{*}$. Similarly, a direct verification reveals that $$f^{******}(Z^{**},W^{*},X^{**})=h_{1}^{***}(g^{******}(Z^{**},W^{*},X^{**})).$$ The weak compactness of $h_{1}$ implies that $f^{******}(Z^{**},W^{*},X^{**})\subseteq Y^{*}$. Now Theorem \ref{theorem7} follows that $f$ is completely regular.
\end{proof}
\begin{example}\label{example2} Let $G$ be a locally compact group and let $M(G)$ be measure algebra of $G$, see \cite[Section 2.5]{foland}.  Let the convolution for $\mu_{1},\mu_{2}\in M(G)$ defined by 
$$\int \psi d(\mu_{1}*\mu_{2})=\int\int \psi(xy)d\mu_{1}(x)d\mu_{2}(y),\ \ (\psi\in C_{0}(G)).$$
 We define the bounded tri-linear mapping $$f:M(G)\times M(G)\times M(G)\longrightarrow M(G)$$
by $f(\mu_{1},\mu_{2},\mu_{3})=\int \psi d(\mu_{1}*\mu_{2}*\mu_{3})$ for $\mu_{1},\mu_{2}$ and $\mu_{3}\in M(G)$. If $G$ is finite,  then $f$ is completely regular.
\end{example}

\section{ The fourth adjoint of a tri-derivation}
\label{se5}

Let $A$  be a Banach algebra and $\pi_{1}:A\times X \longrightarrow X$ be a bounded bilinear map. The pair $(\pi_{1},X)$ is said to be a left Banach $A-$module when $\pi_{1}(\pi_1(a,b),x)=\pi_{1}(a,\pi_{1}(b,x))$, for every $a,b\in A$ and $x\in X$.  A right Banach $A-$module may is defined similarly. Let  $\pi_{2}:X\times A \longrightarrow X$ be a bounded bilinear map. The pair $(X,\pi_{2})$ is said to be a right Banach $A-$module if $\pi_{2}(x,\pi_2(a,b))=\pi_{2}(\pi_{2}(x,a),b)$. A triple $(\pi_{1},X,\pi_{2})$ is said to be a Banach $A-$module if  $(X,\pi_{1})$ and $(X,\pi_{2})$ are left and right Banach $A-$modules, respectively, and $\pi_{1}(a,\pi_{2}(x,b))=\pi_{2}(\pi_1(a,x),b)$. (see, for instance, \cite{Mohamadzadeh} and \cite{Eshaghi})
\begin{definition}\label{definition2}
Let $(\pi_{1},X,\pi_2)$ be a  Banach $A-$module.  A tri-derivation from $A\times A\times A$ to $X$(or $X^{*}$) is a bounded tri-linear map $D:A\times A\times A\longrightarrow X$ satisfying for every $a, b,c,d \in A$,
 \begin{eqnarray*}
D(\pi (a,d),b,c)&=&\pi_{2}(D(a,b,c),d)+\pi_{1}(a,D(d,b,c)),\\
D(a,\pi (b,d),c)&=&\pi_{2}(D(a,b,c),d)+\pi_{1}(b,D(a,d,c)),\\
D(a,b,\pi(c,d))&=&\pi_{2}(D(a,b,c),d)+\pi_{1}(c,D(a,b,d)).
\end{eqnarray*}
\end{definition}
In the section, the bounded trilinear maps $\varphi_{a}$ and $\psi_{x^{*}}$ for every $a,b,c,d\in A$ and $x^{*}\in X^{*}$ define by
 \begin{eqnarray*}
&&\varphi_{a}:A\times A\times A\longrightarrow X \ \ \ \ \ \ \ \ \ \ ,\ \ \ \ \ \ \psi_{x^{*}}:A\times A\times A\longrightarrow A^{*}\\
&&\ \ (c,b,d)\longrightarrow \pi_2(D(a,b,c),d)\ \ \ \ \ \ \ \ \ \ (a,d,b)\longrightarrow D^{*}(\pi_{1}^{*}(x^{*},b),a,d)
\end{eqnarray*}
A standard argument applies to show that
\begin{enumerate}
\item If $\pi_{2}$ is Arens regular, then $\varphi_{a}^{****}=\varphi_{a}^{t****s}$ and $\varphi_{a}^{r****r}=\varphi_{a}^{i****i}$,

\item If $D^{j****j}=D^{****}$, then $\varphi_{a}^{****}=\varphi_{a}^{i****i}$ and $\varphi_{a}^{r****r}=\varphi_{a}^{t****s}$,

\item If $D^{j****j}=D^{i****i}$ or $D^{j****j}=D^{s****t}$, then $\varphi_{a}^{****}=\varphi_{a}^{i****i}$,

\item If $\pi_{2}$ is Arens regular and $D^{j****j}=D^{****}$, then $\varphi_{a}^{****}=\varphi_{a}^{r****r}$.
\end{enumerate}
For the $\psi_{x^{*}}$ have the same argument.
In the next result we provide a necessary and sufficient condition such that the fourth adjoint $D^{****}$  of a tri-derivation is again a tri-derivation.
\begin{theorem}\label{theorem8}
 Let $(\pi_{1},X,\pi_{2})$  be a Banach $A-$module and let $D:A\times A\times A\longrightarrow X$  be a tri-derivation.  Then
$D^{****}:(A^{**},\square)\times(A^{**},\square)\times(A^{**},\square)\longrightarrow X^{**}$ is a tri-derivation if and only if $\varphi_{a}^{r****r}=\varphi_{a}^{i****i}=\varphi_{a}^{s****t}$ and $\psi_{x^{*}}^{j****j}=\psi_{x^{*}}^{t****s}=\psi_{x^{*}}^{****}$. 
\end{theorem}
\begin{proof}
First we show that $D^{****}$ is a tri-derivation. Let $\{a_{\alpha} \}, \{b_{\beta} \}, \{c_{\gamma} \}$ and $\{d_{\tau} \}$ are bounded nets in $A$ which converge to $a^{**}, b^{**}, c^{**}$  and $d^{**}$ in $A^{**}$  in the $w^{*}-$topologies, respectively. If $\varphi_{a}^{r****r}=\varphi_{a}^{i****i}$, then we have 
 \begin{eqnarray*}
\lim\limits_\alpha\lim\limits_\tau\lim\limits_\beta\lim\limits_\gamma \langle x^{*},\pi_{2}(D(a_{\alpha},b_{\beta},c_{\gamma}),d_{\tau}) \rangle&=&\lim\limits_\alpha\lim\limits_\tau\lim\limits_\beta\lim\limits_\gamma \langle x^{*},\varphi_{a_{\alpha}}(c_{\gamma},b_{\beta},d_{\tau})\rangle\\
&=&\lim\limits_\alpha\langle \varphi_{a_{\alpha}}^{r****r}(c^{**},b^{**},d^{**}),x^{*} \rangle\\
&=&\lim\limits_\alpha\langle \varphi_{a_{\alpha}}^{i****i}(c^{**},b^{**},d^{**}),x^{*} \rangle\\
&=&\lim\limits_\alpha\lim\limits_\beta\lim\limits_\gamma\lim\limits_\tau\langle  x^{*},\varphi_{a_{\alpha}}(c_{\gamma},b_{\beta},d_{\tau})\rangle\\
&=&\lim\limits_\alpha\lim\limits_\beta\lim\limits_\gamma\lim\limits_\tau\langle  x^{*},\pi_{2}(D(a_{\alpha},b_{\beta},c_{\gamma}),d_{\tau})\rangle\\
&=&\langle \pi_{2}^{***}(D^{****}(a^{**},b^{**},c^{**}),d^{**}),x^{*}\rangle.
\end{eqnarray*}
Therefore $w^{*}-\lim\limits_\alpha\lim\limits_\tau\lim\limits_\beta\lim\limits_\gamma \pi_{2}(D(a_{\alpha},b_{\beta},c_{\gamma}),d_{\tau})=\pi_{2}^{***}(D^{****}(a^{**},b^{**},c^{**}),d^{**})$. So
\begin{eqnarray*}
&&\langle D^{****}(\pi^{***}(a^{**},d^{**}),b^{**},c^{**}),x^{*}\rangle=\lim\limits_\alpha\lim\limits_\tau\lim\limits_\beta\lim\limits_\gamma \langle x^{*},D(\pi(a_{\alpha},d_{\tau}),b_{\beta},c_{\gamma})\rangle\\
&&=\lim\limits_\alpha\lim\limits_\tau\lim\limits_\beta\lim\limits_\gamma \langle x^{*},\pi_{2}(D(a_{\alpha},b_{\beta},c_{\gamma}),d_{\tau})+\pi_{1}(a_{\alpha},D(d_{\tau},b_{\beta},c_{\gamma}))\rangle\\
&&=\lim\limits_\alpha\lim\limits_\tau\lim\limits_\beta\lim\limits_\gamma \langle x^{*},\pi_{2}(D(a_{\alpha},b_{\beta},c_{\gamma}),d_{\tau})\rangle\\
 &&+\lim\limits_\alpha\lim\limits_\tau\lim\limits_\beta\lim\limits_\gamma\langle x^{*},\pi_{1}(a_{\alpha},D(d_{\tau},b_{\beta},c_{\gamma}))\rangle\\
&&=\langle \pi_{2}^{***}(D^{****}(a^{**},b^{**},c^{**}),d^{**}),x^{*}\rangle+\langle \pi_{1}^{***}(a^{**},D^{****}(d^{**},b^{**},c^{**})),x^{*}\rangle\\
&&=\langle \pi_{2}^{***}(D^{****}(a^{**},b^{**},c^{**}),d^{**})+\pi_{1}^{***}(a^{**},D^{****}(d^{**},b^{**},c^{**})),x^{*}\rangle.
\end{eqnarray*}
Thus
\begin{eqnarray*}
 D^{****}(\pi^{***}(a^{**},d^{**}),b^{**},c^{**})&=&\pi_{2}^{***}(D^{****}(a^{**},b^{**},c^{**}),d^{**})\\
&+&\pi_{1}^{***}(a^{**},D^{****}(d^{**},b^{**},c^{**})).
\end{eqnarray*}
Using the equality $\varphi_{a}^{s****t}=\varphi_{a}^{i****i}$ we have
\begin{eqnarray*}
w^{*}-\lim\limits_\alpha\lim\limits_\beta\lim\limits_\tau\lim\limits_\gamma  \pi_{2}(D(a_{\alpha},b_{\beta},c_{\gamma}),d_{\tau})&=&w^{*}-\lim\limits_\alpha\lim\limits_\beta\lim\limits_\tau\lim\limits_\gamma \varphi_{a_{\alpha}}(c_{\gamma},b_{\beta},d_{\tau})\\
&=&w^{*}-\lim\limits_\alpha \varphi_{a_{\alpha}}^{s****t}(c^{**},b^{**},d^{**})\\
&=&w^{*}-\lim\limits_\alpha \varphi_{a_{\alpha}}^{i****i}(c^{**},b^{**},d^{**})\\
&=&w^{*}-\lim\limits_\alpha\lim\limits_\beta\lim\limits_\gamma \lim\limits_\tau \varphi_{a_{\alpha}}(c_{\gamma},b_{\beta},d_{\tau})\\
&=&\pi_{2}^{***}(D^{****}(a^{**},b^{**},c^{**}),d^{**}).
\end{eqnarray*}
In the other hand, $\psi_{x^{*}}^{j****j}=\psi_{x^{*}}^{t****s}$ implies that
\begin{eqnarray*}
&&\lim\limits_{\alpha}\lim\limits_{\beta}\lim\limits_{\tau}\lim\limits_{\gamma} \langle x^{*},\pi_{1}(b_{\beta},D(a_{\alpha},d_{\tau},c_{\gamma}))\rangle\\
&&=\lim\limits_\alpha\lim\limits_\beta\lim\limits_\tau\lim\limits_\gamma \langle \pi_{1}^{*}(x^{*},b_{\beta}),D(a_{\alpha},d_{\tau},c_{\gamma})\rangle\\
&&=\lim\limits_{\alpha}\lim\limits_{\beta}\lim\limits_{\tau}\lim\limits_{\gamma} \langle D^{*}({\pi_{1}}^{*}(x^{*},b_{\beta}),a_{\alpha},d_{\tau}),c_{\gamma}\rangle\\
&&=\lim\limits_\alpha\lim\limits_\beta\lim\limits_\tau \langle c^{**},D^{*}(\pi_{1}^{*}(x^{*},b_{\beta}),a_{\alpha},d_{\tau})\rangle\\
&&=\lim\limits_\alpha\lim\limits_\beta\lim\limits_\tau \langle c^{**},\psi_{x^{*}}(a_{\alpha},d_{\tau},b_{\beta})\rangle\\
&&=\langle \psi_{x^{*}}^{j****j}(a^{**},d^{**},b^{**}),c^{**}\rangle\\
&&=\langle \psi_{x^{*}}^{t****s}(a^{**},d^{**},b^{**}),c^{**}\rangle\\
&&=\lim\limits_\beta\lim\limits_\alpha\lim\limits_\tau\langle c^{**},\psi_{x^{*}}(a_{\alpha},d_{\tau},b_{\beta})\rangle\\
&&=\lim\limits_\beta\lim\limits_\alpha\lim\limits_\tau\langle c^{**},D^{*}(\pi_{1}^{*}(x^{*},b_{\beta}),a_{\alpha},d_{\tau})\rangle\\
&&=\langle \pi_{1}^{***}(b^{**},D^{****}(a^{**},d^{**},c^{**})),x^{*}\rangle.
\end{eqnarray*}
Therefore we have
\begin{eqnarray*}
\langle D^{****}(a^{**},\pi^{***}(b^{**},d^{**}),c^{**}),x^{*}\rangle&=&\lim\limits_{\alpha}\lim\limits_{\beta}\lim\limits_{\tau}\lim\limits_{\gamma} \langle x^{*},D(a_{\alpha},\pi(b_{\beta},d_{\tau}),c_{\gamma})\rangle\\
&=&\lim\limits_{\alpha}\lim\limits_{\beta}\lim\limits_{\tau}\lim\limits_{\gamma} \langle x^{*},\pi_{2}(D(a_{\alpha},b_{\beta},c_{\gamma}),d_{\tau})\rangle\\
&+&\lim\limits_{\alpha}\lim\limits_{\beta}\lim\limits_{\tau}\lim\limits_{\gamma} \langle x^{*},\pi_{1}(b_{\beta},D(a_{\alpha},d_{\tau},c_{\gamma}))\rangle\\
&=&\langle \pi_{2}^{***}(D^{****}(a^{**},b^{**},c^{**}),d^{**}),x^{*}\rangle\\
&+&\langle \pi_{1}^{***}(b^{**},D^{****}(a^{**},d^{**},c^{**})),x^{*}\rangle\\
&=&\langle \pi_{2}^{***}(D^{****}(a^{**},b^{**},c^{**}),d^{**})\\
&+&\pi_{1}^{***}(b^{**},D^{****}(a^{**},d^{**},c^{**})),x^{*}\rangle.
\end{eqnarray*}
Thus 
\begin{eqnarray*}
D^{****}(a^{**},\pi^{***}(b^{**},d^{**}),c^{**})&=&\pi_{2}^{***}(D^{****}(a^{**},b^{**},c^{**}),d^{**})\\
&+&\pi_{1}^{***}(b^{**},D^{****}(a^{**},d^{**},c^{**})).
\end{eqnarray*}
Now we show that
\begin{eqnarray*}
D^{****}(a^{**},b^{**},\pi^{***}(c^{**},d^{**}))&=&\pi_{2}^{***}(D^{****}(a^{**},b^{**},c^{**}),d^{**})\\
&+&\pi_{1}^{***}(c^{**},D^{****}(a^{**},b^{**},d^{**})).
\end{eqnarray*}
Using the equality $\psi_{x^{*}}^{t****s}=\psi_{x^{*}}^{****}$ we have
\begin{eqnarray*}
&&\lim\limits_{\alpha}\lim\limits_{\beta}\lim\limits_{\gamma}\lim\limits_{\tau} \langle x^{*},\pi_{1}(c_{\gamma},D(a_{\alpha},b_{\beta},d_{\tau}))\rangle\\
&&=\lim\limits_{\alpha}\lim\limits_{\beta}\lim\limits_{\gamma}\lim\limits_{\tau} \langle \pi_{1}^{*}(x^{*},c_{\gamma}),D(a_{\alpha},b_{\beta},d_{\tau})\rangle\\
&&=\lim\limits_{\alpha}\lim\limits_{\beta}\lim\limits_{\gamma}\lim\limits_{\tau} \langle D^{*}(\pi_{1}^{*}(x^{*},c_{\gamma}),a_{\alpha},b_{\beta}),d_{\tau}\rangle\\
&&=\lim\limits_{\alpha}\lim\limits_{\beta}\lim\limits_{\gamma}\langle d^{**},D^{*}(\pi_{1}^{*}(x^{*},c_{\gamma}),a_{\alpha},b_{\beta})\rangle\\
&&=\lim\limits_{\alpha}\lim\limits_{\beta}\lim\limits_{\gamma}\langle d^{**},\psi_{x^{*}}(a_{\alpha},b_{\beta},c_{\gamma})\rangle\\
&&=\langle \psi_{x^{*}}^{****}(a^{**},b^{**},c^{**}),d^{**}\rangle\\
&&=\langle \psi_{x^{*}}^{t****s}(a^{**},b^{**},c^{**}),d^{**}\rangle\\
&&=\lim\limits_{\gamma}\lim\limits_{\alpha}\lim\limits_{\beta}\langle d^{**},\psi_{x^{*}}(a_{\alpha},b_{\beta},c_{\gamma})\rangle\\
&&=\lim\limits_{\gamma}\lim\limits_{\alpha}\lim\limits_{\beta}\langle d^{**},D^{*}(\pi_{1}^{*}(x^{*},c_{\gamma}),a_{\alpha},b_{\beta})\rangle\\
&&=\langle \pi_{1}^{***}(c^{**},D^{****}(a^{**},b^{**},d^{**})),x^{*}\rangle.
\end{eqnarray*}
So it follows that
\begin{eqnarray*}
\langle D^{****}(a^{**},b^{**},\pi^{***}(c^{**},d^{**})),x^{*}\rangle&&=\lim\limits_{\alpha}\lim\limits_{\beta}\lim\limits_{\gamma}\lim\limits_{\tau} \langle x^{*},D(a_{\alpha},b_{\beta},\pi(c_{\gamma},d_{\tau}))\rangle\\
&&=\lim\limits_{\alpha}\lim\limits_{\beta}\lim\limits_{\gamma}\lim\limits_{\tau} \langle x^{*},\pi_{2}(D(a_{\alpha},b_{\beta},c_{\gamma}),d_{\tau})\rangle\\
&&+\lim\limits_{\alpha}\lim\limits_{\beta}\lim\limits_{\gamma}\lim\limits_{\tau} \langle x^{*},\pi_{1}(c_{\gamma},D(a_{\alpha},b_{\beta},d_{\tau}))\rangle\\
&&=\langle \pi_{2}^{***}(D^{****}(a^{**},b^{**},c^{**}),d^{**}),x^{*}\rangle\\
&&+\langle \pi_{1}^{***}(c^{**},D^{****}(a^{**},b^{**},d^{**})),x^{*}\rangle\\
&&=\langle \pi_{1}^{***}(c^{**},D^{****}(a^{**},b^{**},d^{**}))\\
&&+\pi_{1}^{***}(c^{**},D^{****}(a^{**},b^{**},d^{**})),x^{*}\rangle.
\end{eqnarray*}
For the converse,  We prove only $\varphi_{a}^{r****r}=\varphi_{a}^{i****i}$, the other part has the same argument. let $D^{****}$ be a tri-derivation. Then we have 
\begin{eqnarray*}
D^{****}(\pi^{***}(a^{**},d^{**}),b^{**},c^{**})&=&\pi_{2}^{***}(D^{****}(a^{**},b^{**},c^{**}),d^{**})\\
&+&\pi_{1}^{***}(a^{**},D^{****}(d^{**},b^{**},c^{**})).
\end{eqnarray*}
In the other hand
\begin{eqnarray*}
D^{****}(\pi^{***}(a^{**},d^{**}),b^{**},c^{**})&=&w^{*}-\lim\limits_{\alpha}\lim\limits_{\tau}\lim\limits_{\beta}\lim\limits_{\gamma} D(\pi(a_{\alpha},d_{\tau}),b_{\beta},c_{\gamma})\\
&=&w^{*}-\lim\limits_{\alpha}\lim\limits_{\tau}\lim\limits_{\beta}\lim\limits_{\gamma} \pi_{2}(D(a_{\alpha},b_{\beta},c_{\gamma}),d_{\tau})\\
&+&\pi_{1}^{***}(a^{**},D^{****}(d^{**},b^{**},c^{**})).
\end{eqnarray*}
Therefore 
$$\pi_{2}^{***}(D^{****}(a^{**},b^{**},c^{**}),d^{**})=w^{*}-\lim\limits_{\alpha}\lim\limits_{\tau}\lim\limits_{\beta}\lim\limits_{\gamma} \pi_{2}(D(a_{\alpha},b_{\beta},c_{\gamma}),d_{\tau}).$$
In particular,
$$\pi_{2}^{***}(D^{****}(a^{**},b^{**},c^{**}),d^{**})=w^{*}-\lim\limits_{\tau}\lim\limits_{\beta}\lim\limits_{\gamma} \pi_{2}(D(a,b_{\beta},c_{\gamma}),d_{\tau}).$$
Thus we have
\begin{eqnarray*}
\langle \varphi_{a}^{r****r}(c^{**},b^{**},d^{**}),x^{*}\rangle &=&\lim\limits_{\tau}\lim\limits_{\beta}\lim\limits_{\gamma}\langle \varphi_{a}(c_{\gamma},b_{\beta},d_{\tau}),x^{*}\rangle\\
&=&\lim\limits_{\tau}\lim\limits_{\beta}\lim\limits_{\gamma}\langle \pi_{2}(D(a,b_{\beta},c_{\gamma}),d_{\tau}),x^{*}\rangle\\
&=&\langle \pi_{2}^{***}(D^{****}(a,b^{**},c^{**}),d^{**}),x^{*}\rangle\\
&=&\lim\limits_{\beta}\lim\limits_{\gamma}\lim\limits_{\tau}\langle \pi_{2}(D(a,b_{\beta},c_{\gamma}),d_{\tau}),x^{*}\rangle\\
&=&\lim\limits_{\beta}\lim\limits_{\gamma}\lim\limits_{\tau}\langle\varphi_{a}(c_{\gamma},b_{\beta},d_{\tau}),x^{*}\rangle\\
&=&\langle\varphi_{a}(c^{**},b^{**},d^{**}),x^{*}\rangle.
\end{eqnarray*}
Therefore $\varphi_{a}^{r****r}=\varphi_{a}^{i****i}$.
\end{proof}
As an immediate consequence of Theorem \ref{theorem8}, we deduce the next result.
\begin{corollary}\label{corollary9}
If the maps $\varphi_{a}$ and $\psi_{x^{*}}$ are completely regular, then fourth adjoint $D^{****}:(A^{**},\odot)\times(A^{**}\odot)\times(A^{**},\odot)\longrightarrow X^{**}$ is a tri-derivation where $\odot\in \{\square , \lozenge\}$.
\end{corollary}




\end{document}